\documentclass[12pt]{article}

\usepackage{amsmath}
\usepackage{amssymb}
\usepackage{amsthm}
\usepackage{latexsym}
\usepackage{cite}
\usepackage{psfrag}
\usepackage{epsfig}
\usepackage{graphicx}
\usepackage[latin1]{inputenc}

\newtheorem{theorem}{Theorem}[section]

\newtheorem{conjecture}[theorem]{Conjecture}
\newtheorem{corollary}[theorem]{Corollary}
\newtheorem{lemma}[theorem]{Lemma}
\newtheorem{observation}[theorem]{Observation}
\theoremstyle{definition}

\newtheorem{remark}[theorem]{Remark}

\makeatletter \@addtoreset{equation}{section} \makeatother

\def\PG{\mathrm{PG}}
\def\AG{\mathrm{AG}}
\def\B{\mathcal{B}}
\def\C{\mathcal{C}}
\def\K{\mathcal{K}}
\def\M{\mathcal{M}}
\def\P{\mathcal{P}}
\def\U{\mathcal{U}}
\def\T{\mathcal{T}}

\begin{document}
\title{Upper bounds on the smallest size of a complete cap in $\mathrm{PG}(N,q)$, $N\ge3$, under a certain probabilistic
conjecture\thanks{
 The research of A.A.~Davydov was carried out at the IITP RAS at the expense of the Russian
Foundation for Sciences (project 14-50-00150). The research of G. Faina, S. Marcugini and F.
Pambianco was
 supported in part by Ministry for Education, University
and Research of Italy (MIUR) (Project ``Geometrie di Galois e
strutture di incidenza'')
 and by the Italian National Group for Algebraic and Geometric Structures and their Applications
 (GNSAGA - INDAM).}}

\date{}
\author{ Alexander A.
Davydov\\
{\footnotesize Institute for Information Transmission Problems
(Kharkevich
institute)}\\ {\footnotesize Russian Academy of Sciences,
 Bol'shoi Karetnyi per. 19, GSP-4, Moscow,
127994}\\{\footnotesize Russian Federation. E-mail: adav@iitp.ru}
\and Giorgio Faina, Stefano Marcugini and Fernanda Pambianco \\
{\footnotesize Dipartimento di Matematica e Informatica,
Universit\`{a}
degli Studi di Perugia, }\\
{\footnotesize Via Vanvitelli~1, Perugia, 06123, Italy.}\\
{\footnotesize E-mail:
\{giorgio.faina,stefano.marcugini,fernanda.pambianco\}@unipg.it}}

\maketitle

\begin{abstract} In the projective space $\mathrm{PG}(N,q)$ over the
Galois field of order $q$, $N\ge3$, an
iterative step-by-step construction of complete caps by adding a new point
on every step is considered.  It is proved that uncovered points
are evenly placed on the space. A natural conjecture on an estimate of the number of new
covered points on every step is done. For a part
of the iterative process, this estimate is proved rigorously.
  Under the conjecture mentioned,
new upper bounds on the smallest size
$t_{2}(N,q)$ of a complete cap in $\mathrm{PG}(N,q)$ are obtained, in particular,
\begin{align*}
t_{2}(N,q)<\frac{\sqrt{q^{N+1}}}{q-1}\left(\sqrt{(N+1)\ln q}+1\right)+2\thicksim q^\frac{N-1}{2}\sqrt{(N+1)\ln q},\quad N\ge3.
\end{align*}
A connection with
the Birthday problem is noted.  The effectiveness of the
new bounds is illustrated by comparison with
sizes of complete caps obtained by computer in wide regions of $q$.
\end{abstract}



\maketitle
\section{Introduction}
\label{sec_Intro}

Let $\PG(N,q)$ be the $N$-dimensional projective space over the
Galois field $\mathbb{F}_q$ of order $q$. A $k$-cap in  $\PG(N,q)$ is a
set of $k$ points no three of which are collinear.  A $k$-cap $\K$
is complete if it is not contained in a $(k+1)$-cap or, equivalently, if every point of $\PG(N,q) \setminus
\K$ is collinear with two points of $\K$.
Caps in $\PG(2,q)$ are also called arcs and they have been
widely studied by many authors in the past decades, see
\cite{BDFKMP-PIT2014,BDFKMP-JG2015-2,BDFMP-DM,BDFMP-JG2015,%
DFMP-JG2009,GiuliettiSurvey,HirsSt-old,HirsStor-2001,HirsThas-2015,KV,szoT93} and the references therein.
Let $\AG(N,q)$ be the $N$-dimensional affine space over $\mathbb{F}_q$. If $N>2$ only few constructions and bounds are known for small complete caps in $\PG(N,q)$ and $\AG(N,q)$, see
\cite{ABGP_2014_JCD,ABGP_2015_JAC,AG_2013,BDFMP-OC2013,BDKMP-arXivPG3q4q,BDKMP-PG3q4qENDM,BFG2013,BMP2015,%
BGMP-ScatCap,DFMP-JG2009,DGMP2010,DMP-JG2004,DavOst,FaPasSch_2012,GDT1991,%
GiuliettiAffin,Giulietti2007,GiuliettiSurvey,GiulPast,HirsSt-old,HirsStor-2001,HirsThas-2015,%
PS1996,Platoni,Segre,szoT93} for survey and results.

Caps have been intensively studied for
their connection with Coding Theory \cite{HirsSt-old,HirsStor-2001,LandSt}. A linear $q$-ary code with length $n$,
dimension $k$, and minimum distance~$d$ is denoted by  $[n,k,d]_{q}$. If a parity-check matrix of a linear $q$-ary code
is obtained by taking as
columns the homogeneous coordinates of the points of a cap in
$\PG(N,q)$, then the code has  minimum distance $4$ (with the exceptions of
the complete 5-cap in $\PG(3,2)$ and  11-cap in $\PG(4,3)$ giving rise to the
$[5,1,5]_{2}$
and  $[11,6,5]_{3}$  codes). Complete $n$-caps in
$\PG(N,q)$ correspond to non-extendable
 $[n, n-N-1,4]_{q}$ quasi-perfect codes of covering radius~$2$ \cite{BLP1998,CHLL1997}. If $N=2$ these codes are Minimum Distance Separable (MDS); for $N=3$ they are
Almost MDS since their Singleton defect is equal to 1.
 For fixed $N$,
the covering density of the mentioned codes decreases with decreasing $n$.
So, small complete caps have a better covering quality than the big ones.

Note also that caps are connected with quantum
codes; see e.g. \cite{BMP-quantum,Tonchev}.

In general, a central problem concerning caps is to determine
the spectrum of the possible sizes of complete caps in a given
space; see \cite{HirsSt-old,HirsStor-2001} and the references
therein. Of particular interest for applications to Coding
Theory is the lower part of the spectrum as small
complete caps  correspond to
quasi-perfect linear codes with small
covering density.

Let $t_{2}(N,q)$ be the \emph{smallest size}  of a
complete cap in $\PG(N,q)$.

A hard open problem in the study of projective
spaces is the
determination of $t_{2}(N,q)$. The exact values of $t_2(N,q)$, $N\ge3$,  are
known only for very small $q$. For instance, $t_2(3,q)$ is known
only for $q\leq 7$; see \cite[Tab. 3]{DFMP-JG2009}.

\emph{This work} is devoted to \emph{upper bounds }on
$t_{2}(N,q)$, $N\ge3$.

The trivial lower bound for $t_2(N,q)$ is $\sqrt 2 q^{\frac{N-1}{2}}$.
Constructions of complete caps whose size is close to
this lower bound are known only for the following cases: $q=2$ and $N$ arbitrary; $q=2^m>2$ and $N$ odd; $q$ is even square
\cite{BGMP-ScatCap,DFMP-JG2009,GDT1991,DGMP2010,Giulietti2007,PS1996,Segre}.
Using a modification of the approach of
\cite{KV} for the projective plane, the probabilistic upper bound $$t_2(N,q)<cq^{\frac{N-1}{2}}\log^{300} q,$$ where
$c$ is a constant independent of $q$, has been obtained in \cite{BMP2015}.
Computer assisted results
on small complete caps in $\PG(N,q)$ and $\AG(N,q)$ are given in
\cite{BDFMP-OC2013,BDKMP-arXivPG3q4q,BDKMP-PG3q4qENDM,BFG2013,DMP-JG2004,DFMP-JG2009,FaPasSch_2012,Platoni}.

The main result of the paper is given by Theorem \ref{th1_main bounds} based on Theorem~\ref{th4_mainmain}.
\begin{theorem}\label{th1_main bounds}\emph{(}\textbf{the main result}\emph{)} Let $t_{2}(N,q)$ be the \emph{smallest size}  of a
complete cap in the projective space $\PG(N,q)$. Let $D\ge1$ be a constant independent of~$q$.

\emph{\textbf{(i)}} Under Conjecture \emph{\ref{Conj_main}(i)}, in  $\PG(N,q)$, it holds that
\begin{align}\label{eq1_bnd_mainD}
t_{2}(N,q)<\frac{\sqrt{q^{N+1}}}{q-1}\left(\sqrt{D}\sqrt{(N+1)\ln q}+1\right)+2 \thicksim
\sqrt{D}q^{\frac{N-1}{2}}\sqrt{(N+1)\ln q},\quad N\ge3.
\end{align}

\emph{\textbf{(ii)}} Under Conjecture  \emph{\ref{Conj_main}(ii)}, in  $\PG(N,q)$, the bound \eqref{eq1_bnd_mainD} with $
    D=1$ holds, i.e.
\begin{align}\label{eq1_bnd_mainD=1}
t_{2}(N,q)<\frac{\sqrt{q^{N+1}}}{q-1}\left(\sqrt{(N+1)\ln q}+1\right)+2 \thicksim
q^{\frac{N-1}{2}}\sqrt{(N+1)\ln q},\quad N\ge3.
\end{align}
\end{theorem}
\begin{conjecture}\label{conjec1}
In $\PG(N,q)$, $N\ge3$, the upper bound
\eqref{eq1_bnd_mainD=1}  holds for all $q$ without any extra conditions
and conjectures.
\end{conjecture}

This work can be treated as a development of the paper
\cite{BDFKMP-PIT2014}.

Some results of this work were briefly presented in \cite{BDFMP-Bulg2016ENDM}.

The paper is organized as follows. In Section \ref{sec_iterProc}, we describe the iterative step-by-step process constructing caps.
In Section \ref{sec_probabConj}, probabilities of events, that points of $\PG(N,q)$ are not covered by a running cap, are considered.
It is proved that uncovered points
are evenly placed on the space. A natural Conjecture \ref{Conj_main} on an estimate of the number of new
covered points on every step of the iterative process is done.  In Section \ref{sec_bounds}, under the conjecture of Section \ref{sec_probabConj} we give new upper bounds on
$t_{2}(N,q)$. In Section \ref{sec_effectiveness}, we illustrate the effectiveness of the new bounds comparing them with the results of computer search from the papers \cite{BDKMP-arXivPG3q4q,BDKMP-PG3q4qENDM}. A rigorous proof of Conjecture~\ref{Conj_main} for a part of the iterative process is given in Section \ref{sec_rigorous}. In Section~\ref{sec_reason}, the \emph{reasonableness of Conjecture }\ref{Conj_main} is discussed. It is shown that in the steps of the iterative process when the rigorous estimates give not good results, actually these estimates do not reflect the real
situation effectively. The reason is that the rigorous estimates assume that the number of uncovered points on unisecants is the same for all unisecants.
However, in fact, there is a dispersion of the number of uncovered points on unisecants, see Fig.\ \ref{fig_unisecant_gamma}. Moreover,  this dispersion
grows in the iterative process. In Conclusion, the obtained results are briefly discussed.

\section{An iterative step-by-step process}\label{sec_iterProc}
 Assume that in $\PG(N,q)$, $N\geq 3,$ a complete cap
is constructed by a step-by-step algorithm (\emph{Algorithm} for short)
which adds one new point to the cap in each step. As an example, we can
mention the greedy algorithm that in every step adds to the cap a point
providing the maximal possible (for the given step) number of new covered
points; see \cite%
{BDFMP-DM,BDFMP-JG2015,DFMP-JG2009,DMP-JG2004}.

Recall that a \emph{point} of $\PG(N,q)$ is \emph{covered by} a
\emph{cap} if the point lies on a bisecant of the cap, i.e.\ on a line
meeting the cap in two points. Clearly, all points of the cap are covered.

The space $\PG(N,q)$ contains
\begin{equation*}
\theta_{N,q}=\frac{q^{N+1}-1}{q-1}=q^{N}+q^{N-1}+\ldots +q+1
\end{equation*}
points.

Assume that after the $w$-th step of Algorithm, a $w$-cap
is obtained that does not cover exactly $U_{w}$ points. Let $\mathbf{S}(U_{w})$
be the set of all $w$-caps in $\PG(N,q)$ each of which does not
cover exactly $U_{w}$ points. Evidently, the group of collineations $P\Gamma
L(N+1,q)$ preserves $\mathbf{S}(U_{w})$.

Consider the $(w+1)$-st step of Algorithm. This step starts from a $w$-cap
$\K_{w}$ with $\K_{w}\in \mathbf{S}(U_{w})$.
The choice $\K_{w}$ from $\mathbf{S}(U_{w})$ can be done by distinct ways.

One way is to choose randomly a $w$-cap
of $\mathbf{S}(U_{w})$ so that for every cap of $\mathbf{S}(U_{w})$
the probability to be chosen is equal to $\frac{1}{\#\mathbf{S}(U_{w})}$.
In this case, the set $\mathbf{S}(U_{w})$ is considered as an \emph{ensemble
of random objects} with the uniform probability distribution. Anywhere where we say on probabilities
and mathematical expectations, the such random choice is supposed.

On the other side,
sometimes we study some values average or maximum
by all caps of $\mathbf{S}(U_{w})$ without a random choice. Also, we can consider
some properties that hold for all caps of $\mathbf{S}(U_{w})$.

Finally, for practice calculations (in particular, for the illustration
of investigations) we use the same cap adding to it an one point in the each step
of the iterative process.

Denote by $\U(\K)$ the set of points of $\PG(N,q)$
that are not covered by a cap~$\K$. By the definition,
$$
\#\U(\K_{w})=U_{w}.
$$
Let the cap $\K_{w}$ consist of $w$ points $A_{1},A_{2},\ldots
,A_{w}$. Let $A_{w+1}\in \U(\K_{w})$ be the point that
will be included into the cap in the $(w+1)$-st step.

\begin{remark}
\label{rem2_Aw+1} Below we introduce a few point subsets, depending on $
A_{w+1}$, for which we use the notation of the type $\M_{w}(
A_{w+1})$. Any uncovered point may be added to $\K
_{w}$.  So, there exist $U_{w}$ distinct subsets $\mathcal{
M}_{w}(A_{w+1}).$ When a particular point $A_{w+1}$ is not relevant, one may
use the short notation $\M_{w}.$ The same concerns to quantities $
\Delta _{w}(A_{w+1})$ and $\Delta _{w}$ introduced below.
\end{remark}
A point $A_{w+1}$ defines a bundle $\B(A_{w+1})$ of $w$ unisecants
 to $\K_{w}$ which are denoted as $\overline{A_{1}A_{w+1}
},\overline{A_{2}A_{w+1}},\ldots ,\overline{A_{w}A_{w+1}}$, where $\overline{A_{i}A_{w+1}}$ is the unisecant connecting $A_{w+1}$ with the cap point $A_{i}$. Every unisecant contains $q+1$ points.
Except for $
A_{1},\ldots ,A_{w}$, all the points on the unisecants in the bundle are \textbf{candidates} to
be new covered points in the $(w+1)$-st step. Denote by $\C
_{w}(A_{w+1})$ the point \emph{set of the candidates}.  By the definition,
\begin{align*}
& \mathcal{C}_{w}(A_{w+1})=\B(A_{w+1})\setminus \K_{w}, \\
& \#\mathcal{C}_{w}=w(q-1)+1.
\end{align*}
We call $\{A_{w+1}\}$ and $\B(A_{w+1})\setminus (\K_{w}\cup
\{A_{w+1}\})$, respectively, the \emph{head} and the \emph{basic part} of the bundle $\B
(A_{w+1})$.
For a given cap $\K_{w}$, in total, there
are $\#\U(\K_{w})=U_w$ distinct bundles and, respectively, $U_w$ distinct sets of the candidates.

Let $\Delta _{w}(A_{w+1})$ be the number of \textbf{new covered points} in the $
(w+1)$-st step, i.e.
\begin{equation}
\Delta _{w}(A_{w+1})=\#\U(\K_{w})-\#\U
(\K_{w}\cup
\{A_{w+1}\})=\#\{\C_{w}(A_{w+1})\cap \U(\K
_{w})\}.  \label{eq2_Delta_w}
\end{equation}

In future, we consider continuous approximations of the discrete functions
$\Delta _{w}(A_{w+1}),$ $\#\U(\K_{w}),$ $\#
\U(\K_{w}\cup
\{A_{w+1}\}),$ and some other ones keeping the same notations.

\section{Probabilities of uncovering. Conjectures
on the number of new covered points in every step
}\label{sec_probabConj}

Let $n_{w}(H)$ be the number of caps of $
\mathbf{S}(U_{w})$ that do not cover a point $H$ of $\PG(N,q)$. Each
point $H\in PG(N,q)$ will be considered as a random object that is not
covered by a randomly chosen $w$-cap $\K_{w}$ with some probability
$p_{w}(H)$ defined as
\begin{equation*}
p_{w}(H)=\frac{n_{w}(H)}{\#\mathbf{S}(U_{w})}.
\end{equation*}

\begin{lemma}
\label{lem2_nw} The value $n_{w}(H)$ is the same for all points $H\in
\PG(N,q) $.
\end{lemma}

\begin{proof}
Let $\mathbf{K}_{w}(H)\subseteq \mathbf{S}(U_{w})$ be the subset of $w$-caps
in $\mathbf{S}(U_{w})$ that do not cover~$H$. By the definition, $
n_{w}(H)=\#\mathbf{K}_{w}(H)$. Let $H_{i}$ and $H_{j}$ be two distinct
points of $\PG(N,q)$. In the group $P\Gamma L(N+1,q)$, denote by $\Psi
(H_{i},H_{j})$ the subset of collineations taking $H_{i}$ to $H_{j}$.
Clearly, $\Psi (H_{i},H_{j})$ embeds the subset $\mathbf{K}_{w}(H_{i})$ in $
\mathbf{K}_{w}(H_{j})$. Therefore, $\#\mathbf{K}_{w}(H_{i})\leq \#\mathbf{K}
_{w}(H_{j})$. Vice versa, $\Psi (H_{j},H_{i})$ embeds $\mathbf{K}_{w}(H_{j})$
into $\mathbf{K}_{w}(H_{i}),$ and we have $\#\mathbf{K}_{w}(H_{j})\leq \#
\mathbf{K}_{w}(H_{i})$. Thus, $\#\mathbf{K}_{w}(H_{i})=\#\mathbf{K}
_{w}(H_{j}),$ i.e. $n_{w}(H_{i})=n_{w}(H_{j})$.
\end{proof}

So, $n_{w}(H)$ can be considered as $n_{w}$. This means that the
\emph{probability} $p_{w}(H)$ is \emph{the same for all points} $H$; it may be
considered as
\begin{equation*}
p_{w}=\frac{n_{w}}{\#\mathbf{S}(U_{w})}.
\end{equation*}
In turn, since the probability to be uncovered is independent of a point, we
conclude that, for a $w$-cap $\K_{w}$ randomly chosen from $\mathbf{
S}(U_{w}),$ the \emph{fraction} $\#\U_{w}(\K
_{w})/\theta_{N,q}$ of uncovered points of $\PG(N,q)$ \emph{is equal to the probability} $p_{w}$ that a point of $\PG(N,q)$ is not covered.
In other words,
\begin{equation}
p_{w}=\frac{\#\U_{w}(\K_{w})}{\theta_{N,q}}=\frac{
U_{w}}{\theta_{N,q}}.  \label{eq3_p_w}
\end{equation}

Equality \eqref{eq3_p_w} can also be explained as follows. By Lemma \ref
{lem2_nw}, the multiset consisting of all points that are not covered by all
caps of $\mathbf{S}(U_{w})$ has cardinality $n_{w}\cdot \#PG(N,q)$, where $\#PG(N,q)=\theta_{N,q}$. This
cardinality can also be written as $U_{w}\cdot \#\mathbf{S}(U_{w})$.
Thus, $n_{w}\theta_{N,q}=U_{w}\cdot \#\mathbf{S}(U_{w}),$ whence
\begin{equation*}
\frac{n_{w}}{\#\mathbf{S}(U_{w})}=\frac{U_{w}}{\theta_{N,q}}.
\end{equation*}
Let $s_{w}(h)$ be the number of ones in a sequence of $h$ random and
independent 1/0 trials each of which yields 1 with the probability $p_{w}$. For
the random variable $s_{w}(h)$ we have the \emph{binomial} probability
\emph{distribution}; the \emph{expected value} of $s_{w}(h)$ is
\begin{equation}
\mathbf{E}[s_{w}(h)]=hp_{w}=h\frac{U_{w}}{\theta_{N,q}}.
\label{eq3_expect}
\end{equation}

\begin{remark}
One can consider also the \emph{hypergeometric} probability \emph{distribution}, which describes the probability of $
s_{w}^{\prime }(h)$ successes in $h$ random and independent draws without
replacement from a finite population of size $\theta_{N,q}$ containing
exactly $U_{w}$ successes.  The \emph{expected value} of $
s_{w}^{\prime }(h)$ again is
\begin{equation*}
\mathbf{E}[s_{w}^{\prime }(h)]=h\frac{U_{w}}{\theta_{N,q}}=\mathbf{E}[s_{w}(h)].
\end{equation*}

Note also that the \emph{average number } of
uncovered points among $h$ points of $\PG(N,q)$ calculated over all $\binom{
\theta_{N,q}}{h}$ combinations of $h$ points is
\begin{align*}
&\frac{1}{\binom{\theta_{N,q}}{h}}\sum\limits_{i=1}^{h}i\binom{\theta_{N,q}-U_{w}}{h-i}
\binom{U_{w}}{i}=\frac{U_{w}}{\binom{\theta_{N,q}}{h}}
\sum\limits_{i=1}^{h}\binom{\theta_{N,q}-U_{w}}{h-i}\binom{U_{w}-1
}{i-1}=
\frac{U_{w}\binom{\theta_{N,q}-1}{h-1}}{\binom{\theta_{N,q}}{h}}\\
&
=h\frac{U_{w}}{\theta_{N,q}}=\mathbf{E}[s_{w}(h)].
\end{align*}
\end{remark}

Denote by $\mathbf{E}_{w,q}$
 the \textbf{expected value} of the number of uncovered points among\linebreak  $w(q-1)+1$ \emph{randomly} taken points
in $\PG(N,q)$, if
the events to be  uncovered are  \emph{independent}. By Lemma \ref{lem2_nw}, taking into account \eqref{eq3_p_w},
\eqref{eq3_expect}, we have
\begin{equation}\label{eq3_Ewq}
\mathbf{E}_{w,q} =\mathbf{E}[s_{w}(w(q-1)+1)]=(w(q-1)+1)p_w=\frac{(w(q-1)+1)U_{w}}{
\theta_{N,q}}\,.
\end{equation}

In \eqref{eq2_Delta_w}, we defined $\Delta _{w}(A_{w+1})$ as the number of new covered points
on the $(w+1)$-st step. Since all candidates to be new covered points lie on some bundle, they
cannot be considered as randomly taken points for which the events to be  uncovered are  independent.
So, in the general case, the expected value $\mathbf{E}
[\Delta _{w}]$ is not equal to $\mathbf{E}_{w,q} $.

On the other side,  there is a large number of random factors affecting the process,
for instance, the relative positions and intersections of bisecants and unisecants.
These factors especially act for growing $q$, when the volume of the
ensemble $\mathbf{S}(U_{w})$ and the number of distinct bundles $
\B(A_{w+1})$ are relatively large. Therefore, the variance of the
random variable $\Delta _{w},$ in principle, implies the existence of
bundles $\B(A_{w+1})$ providing the inequality $\Delta _{w}(A_{w+1})>\mathbf{E}
[\Delta _{w}]$.  By these arguments (see also Section \ref{sec_reason}),
Conjecture \ref{Conj_main} seems to be reasonable and
founded.

\begin{conjecture}
\label{Conj_main}\emph{\textbf{(i)}} \emph{(}\textbf{the generalized conjecture}\emph{)}
 In $\PG(N,q),$
for $q$ large enough, in every $(w+1)$-st step of the iterative process, considered in Section~\emph{\ref{sec_iterProc}}, there exists a $w$-cap $\K
_{w}\in \mathbf{S}(U_{w})$ such that one can find an uncovered point
$A_{w+1}$ providing the inequality
\begin{align}
\Delta _{w}(A_{w+1})\geq \frac{\mathbf{E}_{w,q}}{D}=\frac{1}{D}\cdot\frac{(w(q-1)+1)U_{w}}{
\theta_{N,q}}, \label{eq3_Delta>E/D}
\end{align}
where  $D\ge1$ is a constant independent of~$q$.

  \emph{\textbf{(ii)}} \emph{(}\textbf{the basic conjecture}\emph{)} In \eqref{eq3_Delta>E/D} we have $D=1$.
\end{conjecture}

\section{Upper bounds on $t_2(N,q)$}\label{sec_bounds}

We denote
\begin{equation}
Q=\frac{\theta_{N,q}}{q-1}=\frac{q^{N+1}-1}{(q-1)^2}.  \label{eq4_Q}
\end{equation}
By Conjecture \ref{Conj_main},
 taking into account~\eqref{eq2_Delta_w}, \eqref{eq3_Ewq}, \eqref{eq3_Delta>E/D}, we obtain
\begin{align}
& \#\U(\K_{w}\cup\{A_{w+1}\})=\#\U(\K_{w})-\Delta _{w}(A_{w+1})
  \label{eq4_|N_w+1|} \\
&\le U_w\left( 1-\frac{w(q-1)+1}{D\theta_{N,q}}
\right) <U_w\left( 1-\frac{w(q-1)}{D\theta_{N,q}}\right) < U_w\left( 1-\frac{w}{DQ}\right).  \notag
\end{align}
Clearly, $\#\U(\K_{1})=U_1=\theta_{N,q}-1.$ Using (\ref
{eq4_|N_w+1|}) iteratively, we have
\begin{equation}
\#\U(\K_{w}\cup\{A_{w+1}\})\leq (\theta_{N,q}-1)f_{q}(w;D)<\theta_{N,q}f_{q}(w;D)  \label{eq4_N_w+1<}
\end{equation}
where
\begin{equation}
f_{q}(w;D)=\prod_{i=1}^{w}\left( 1-\frac{i}{DQ}\right) .  \label{eq4_f(w)}
\end{equation}

\begin{remark}
\label{rem2_birthday} The function $f_{q}(w;D)$ and its approximations,
including (\ref{eq4_fqw_approx}), appear in distinct
tasks of Probability Theory, e.g. in the \emph{Birthday problem} (or the
Birthday paradox) \cite{BrinkBirthday,BirthMajor,BirthRevis}. Really,
let the year contain $DQ$ days and let all birthdays occur with the same
probability. Then $P_{DQ}^{\neq }(w+1)=f_{q}(w;D)$ where $P_{DQ}^{\neq }(w+1)$
is the probability that no two persons from $w+1$ random persons have the
same birthday. Moreover, if birthdays occur with different probabilities we
have $P_{DQ}^{\neq }(w+1)<f_{q}(w;D)$ \cite{BirthMajor}.
\end{remark}

In further, we consider a \emph{truncated iterative process}. The iterative process ends when $\#\U(\K_{w}\cup\{A_{w+1}\})\leq \xi $ where $\xi
\geq 1$ is some value chosen to improve estimates. Then a few (at most $\xi $
) points are added to $\K_{w}$ in order to get a complete $k$-cap.
The size $k$ of an obtained complete cap is as follows:
\begin{equation}
w+1\leq k\leq w+1+\xi \text{ under condition }\#\U(\K_{w}\cup\{A_{w+1}\})\leq \xi .
\label{eq4_k=w+1+xi}
\end{equation}

\begin{theorem}
\label{th4_main} Let $f_{q}(w;D)$ be as in
\eqref{eq4_f(w)}. Let $\xi$ be a constant independent of $w$ with $\xi \geq 1$. Under Conjecture
 \emph{\ref{Conj_main}}, in $\PG(N,q)$ it holds
that
\begin{equation}
t_{2}(N,q)\leq w+1+\xi  \label{eq4_general bound}
\end{equation}
where the value $w$ satisfies the
 inequality
\begin{equation}
f_{q}(w;D)\leq \frac{\xi }{\theta_{N,q}}.  \label{eq4_main}
\end{equation}
\end{theorem}

\begin{proof}
By \eqref{eq4_N_w+1<}, to provide the inequality $\#\U(\K_{w}\cup\{A_{w+1}\})\leq
\xi $ it is sufficient to find $w$ such that $\theta_{N,q}f_{q}(w;D)\leq \xi $.
Now \eqref{eq4_general
bound} follows from \eqref{eq4_k=w+1+xi}.
\end{proof}

We find an upper bound on the smallest possible solution of inequality (\ref
{eq4_main}).

The Taylor series of $e^{-\alpha }$ implies $
1-\alpha <e^{-\alpha }\mbox{ for }\alpha \neq 0$, whence
\begin{align}
&\prod_{i=1}^{w}\left( 1-\frac{i}{DQ}\right) <\prod_{i=1}^{w}e^{-i/DQ}=e^{-(w^{2}+w)/2DQ}<e^{-w^{2}/2DQ}.\label{eq4_fqw_approx}
\end{align}

\begin{lemma}
\label{lem2_basic}  Let $\xi$ be a constant independent of $w$ with $\xi \geq 1$. The value
\begin{equation}
w\geq \sqrt{2DQ}\sqrt{\ln \frac{\theta_{N,q}}{\xi }}+1  \label{eq4_w<}
\end{equation}
satisfies the   inequality \eqref{eq4_main}.
\end{lemma}

\begin{proof}
By (\ref{eq4_f(w)}),(\ref{eq4_fqw_approx}), to provide (\ref{eq4_main}) it
is sufficient to find $w$ such that
\begin{equation*}
e^{-w^{2}/2DQ}\leq \frac{\xi }{\theta_{N,q}}.
\end{equation*}
As $w$ should be an integer, in \eqref{eq4_w<} one is added.
\end{proof}

\begin{theorem}
\label{th4_general bound_converse} Let $D\ge1$ be a constant independent of~$q$. Under Conjecture \emph{\ref{Conj_main}(i)}, in $
PG(N,q)$ it holds that
\begin{equation}
t_{2}(N,q)\leq \sqrt{2DQ}\sqrt{\ln \frac{
\theta_{N,q}}{\xi }}+\xi+2 ,~~\xi \geq 1,  \label{eq4_gen bound}
\end{equation}
where $\xi $ is an arbitrarily chosen constant independent of $w$.
\end{theorem}

\begin{proof}
The assertion follows from \eqref{eq4_general bound} and \eqref{eq4_w<}.
\end{proof}

We should choose $\xi $ so  to obtain a relatively small value in the
right part of~\eqref{eq4_gen bound}.
 We consider
the function of $\xi$ of the form
\begin{equation*}
\phi(\xi)= \sqrt{2DQ}\sqrt{\ln \frac{
\theta_{N,q}}{\xi }}+\xi+2.
\end{equation*}
Its derivative by $\xi$ is
\begin{equation*}
\phi'(\xi)=1-\frac{1}{\xi}\sqrt{\frac{DQ}{2\ln\frac{\theta_{N,q}}{\xi }}}\,\,.  \label{eq2_der phi}
\end{equation*}
Put $\phi'(\xi)=0$. Then
\begin{align} \label{eq4_xi^2}
    \xi^2=\frac{DQ}{2\ln\theta_{N,q}-2\ln \xi}=\frac{D\theta_{N,q}}{2(q-1)(\ln\theta_{N,q}-\ln \xi)}\,\,.
\end{align}
We find $\xi$ in the form
$
\xi=\sqrt{\frac{\theta_{N,q}}{c\ln\theta_{N,q}}}\,\,.$
 By \eqref{eq4_xi^2},
 \begin{align*}
 c=\frac{q-1}{D\ln\theta_{N,q}}(\ln\theta_{N,q}+\ln c+\ln\ln\theta_{N,q})=
   \frac{q-1}{D}\left(1+\frac{\ln c+\ln\ln \theta_{N,q}}{\ln \theta_{N,q}}\right).
 \end{align*}
So, for growing $q$ one could take
\begin{align*}
   c=\frac{q-1}{D},\quad \xi=\sqrt{\frac{D\theta_{N,q}}{(q-1)\ln\theta_{N,q}}}=\sqrt{\frac{D(q^{N+1}-1)}{(q-1)^2\ln\theta_{N,q}}}\,\,.
\end{align*}

For simplicity of the presentation,
we put
\begin{equation}
\xi =\frac{\sqrt{q^{N+1}}}{q-1}\,.  \label{eq4_xi good}
\end{equation}

\begin{theorem}\label{th4_mainmain}  Let $D\ge1$ be a constant independent of~$q$.
 Under Conjecture \emph{\ref{Conj_main}(i)}, the following upper
bound on the smallest size $t_{2}(N,q)$ of a complete cap in $\PG(N,q)$, $N\ge3$,
holds:
\begin{align}\label{eq4_bnd_main}
&t_{2}(N,q)<\frac{\sqrt{q^{N+1}}}{q-1}\left(\sqrt{D}\sqrt{(N+1)\ln q}+1\right)+2\thicksim\sqrt{D}q^{\frac{N-1}{2}}\sqrt{(N+1)\ln q}.
\end{align}
\end{theorem}

\begin{proof}
In (\ref{eq4_gen bound}), we take $Q$ and $\xi$ from \eqref{eq4_Q} and \eqref{eq4_xi good} and obtain
\begin{equation*}
t_{2}(N,q)< \sqrt{2D\frac{q^{N+1}-1}{(q-1)^2}\cdot\ln\frac{\frac{q^{N+1}-1}{q-1}}{\frac{q^{\frac{N+1}{2}}}{q-1}} }
+\frac{\sqrt{q^{N+1}}}{q-1}+2
\end{equation*}
whence  the relation \eqref{eq4_bnd_main} follows directly as $q^{N+1}-1<q^{N+1}$.
\end{proof}

From Theorem \ref{th4_mainmain} we obtain Theorem \ref{th1_main bounds}.

\section{Illustration of the effectiveness of the new bounds}\label{sec_effectiveness}
In the works \cite{BDKMP-arXivPG3q4q,BDKMP-PG3q4qENDM}, for $\PG(N,q)$,  $N=3,4$, $q\in L_N$, complete caps are obtained by computer search.  Here
\begin{align*}
&L_{3}:=\{q\le 4673, ~q\  \textrm{prime}\}  \cup \{5003, 6007, 7001, 8009\},\\
&L_{4}:=\{q\le 1361, ~q\  \textrm{prime}\}\cup \{1409\} .
\end{align*}
\emph{All obtained complete caps \textbf{satisfy} bound} \eqref{eq4_bnd_main} \emph{with} D = 1 (equivalently, \emph{bound} \eqref{eq1_bnd_mainD=1}).

Let $\overline{t}_{2}(N,q)$ be the smallest known size of
complete caps in $\PG(N,q)$; these sizes can be found in  \cite{BDKMP-arXivPG3q4q}.

In  Fig.\,\ref{fig_3_4q_1} we compare the upper bound of \eqref{eq1_bnd_mainD=1}
 with the sizes $\overline{t}_{2}(N,q)$. The top dashed-dotted red curve, corresponding to the bound of \eqref{eq1_bnd_mainD=1}, is \emph{\textbf{strictly higher}} than
the bottom black curve $\overline{t}_{2}(N,q)$.
\begin{figure}[htbp]
\includegraphics[width=\textwidth]{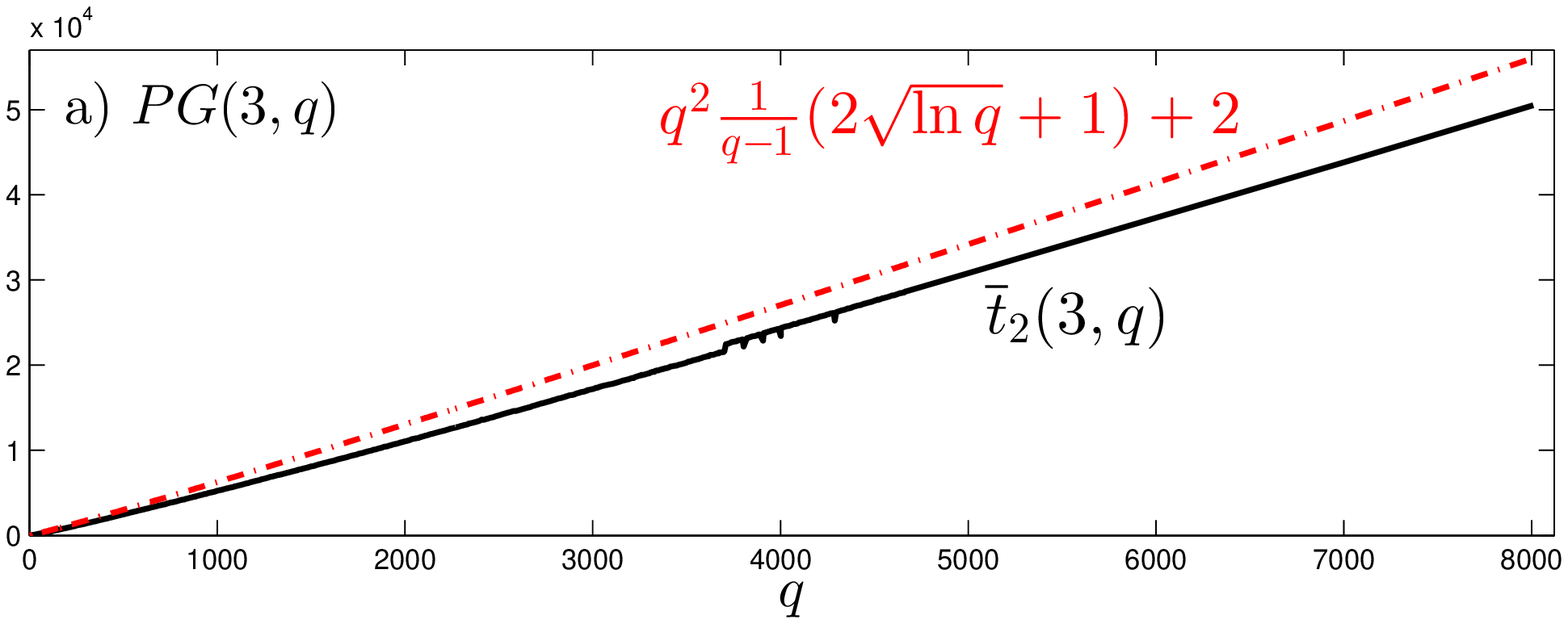}
\includegraphics[width=\textwidth]{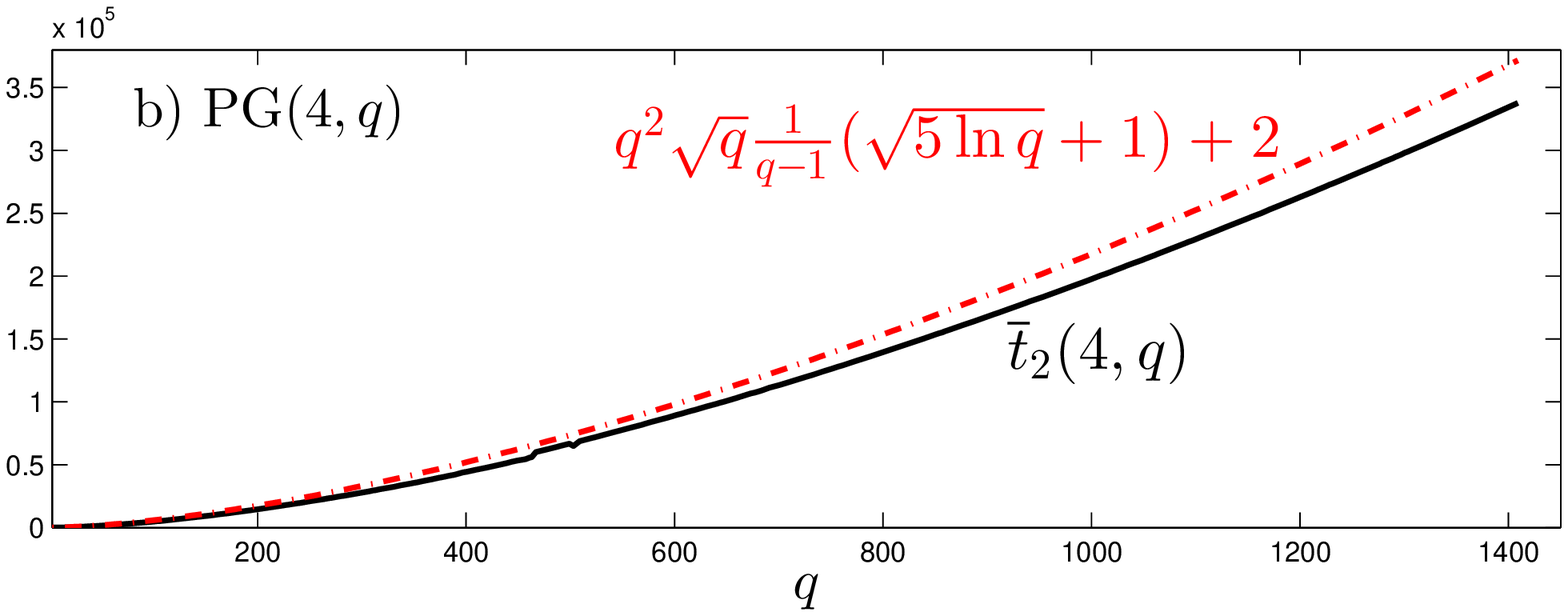}
\caption{\textbf{Bound}  $t_{2}(N,q)<\frac{\sqrt{q^{N+1}}}{q-1}\left(\sqrt{(N+1)\ln q}+1\right)+2$  (\emph{top dashed-dotted red curve}) \textbf{vs the smallest known sizes} $\overline{t}_{2}(N,q)$ of complete caps ,
$q\in L_{N}$, $N=3,4$ (\emph{bottom black curve}). a) $\PG(3,q)$
b) $\PG(4,q)$
 }
\label{fig_3_4q_1}
\end{figure}

\section{A rigorous proof of Conjecture \ref{Conj_main} for a part of the iterative process}
\label{sec_rigorous}

In further, we take into account that \emph{all points that are not
covered} by a cap \emph{lie on unisecants} to the cap.

In total there are  $\theta_{N-1,q}$ lines through every point of $
\PG(N,q).$ Therefore, through
every point $A_i$ of $\K_{w}$ there is a pencil $\P(A_i)$ of $\theta_{N-1,q}-(w-1)$ unisecants to $\K_{w}$, where $i=1,2,\ldots,w$. The total number $T_{w}^{\Sigma}$ of the unisecants to $\K_{w}$ is
\begin{equation}
T_{w}^{\Sigma }=w(\theta_{N-1,q}+1-w).  \label{eq6_total unisecants}
\end{equation}
Let $\gamma _{w,j}$ be \smallskip the number of uncovered points on the $j$-th
unisecant $\T_{j}$, $j=1,2,\ldots ,T_{w}^{\Sigma }.$

\begin{observation}\label{observ}
Every unisecant to $\K_w$ belongs to one and only one pencil $\P(A_i)$, $i\in\{1,2,\ldots,w\}$. Every uncovered point belongs to
one and only one unisecant from every pencil $\P(A_i)$, $i=1,2,\ldots,w$. Every
uncovered point $A$ lies on exactly $w$ unisecants which form the bundle
$\B(A)$ with the head $\{A\}$. All unisecants from the same bundle belong to distinct pencils. A unisecant $\T_{j}$ belongs to $\gamma _{w,j}$ distinct bundles.
\end{observation}

 Every
uncovered point lies on exactly $w$ unisecants; due to this \emph{multiplicity},
on all unisecants there are in total $\Gamma _{w}^{\Sigma }$ uncovered
points, where
\begin{equation}
\Gamma _{w}^{\Sigma }=\sum_{j=1}^{T_{w}^{\Sigma }}\gamma _{w,j}=wU_{w}.
\label{eq6_Gamma}
\end{equation}
By \eqref{eq6_total unisecants}, \eqref{eq6_Gamma},
the average number $\gamma_{w}^{\text{aver}}$  of uncovered points on a unisecant is
\begin{equation}\label{eq6_aver unisecants}
\gamma_{w}^{\text{aver}}=\frac{\Gamma _{w}^{\Sigma }}{T_{w}^{\Sigma }}=\frac{U_{w}}{\theta_{N-1,q}+1-w}.
\end{equation}

A \emph{unisecant $\T_{j}$ belongs to $\gamma _{w,j}$ distinct bundles}, as every
uncovered point on $\T_{j}$ may be the head of a bundle. Moreover, $
\T_{j}$ provides $\gamma _{w,j}(\gamma _{w,j}-1)$ uncovered
points to the basic parts of all these bundles. The noted points are counted with
\emph{multiplicity}.

\emph{Taking into account the
multiplicity}, in all $U_{w}$ the bundles there are
\begin{equation}\label{eq6_sumDeltaw}
\sum\limits_{A_{w+1}}\Delta
_{w}(A_{w+1})=U_{w}+\sum\limits_{j=1}^{T_{w}^{\Sigma }}\gamma _{w,j}(\gamma
_{w,j}-1)
\end{equation}
uncovered points, where $U_{w}$ is the total numbers of all the heads. By \eqref{eq6_Gamma}, \eqref{eq6_sumDeltaw},
\begin{equation*}
\sum\limits_{A_{w+1}}\Delta
_{w}(A_{w+1})=U_{w}+\sum\limits_{j=1}^{T_{w}^\Sigma}\gamma _{w,j}^{2}-\sum\limits_{j=1}^{T_{w}^\Sigma}\gamma_{w,j}=
U_{w}(1-w)+\sum\limits_{j=1}^{T_{w}^\Sigma}\gamma _{w,j}^{2}.
\end{equation*}

For a cap $\K_{w}$, we denote by $\Delta _{w}^{\text{aver}}(\K
_{w})$ the average value of $\Delta _{w}(A_{w+1})$ by all $\#\U(
\K_{w})$ uncovered points $A_{w+1}$, i.e.
\begin{equation}
\Delta _{w}^{\text{aver}}(\K_{w})=\frac{\sum\limits_{A_{w+1}}\Delta
_{w}(A_{w+1})}{\#\U(\K_{w})}=\frac{\sum\limits_{A_{w+1}}\Delta
_{w}(A_{w+1})}{U_{w}}=
\frac{\sum\limits_{j=1}^{T_{w}^\Sigma}\gamma _{w,j}^{2}}{U_w}-w+1\geq 1
\label{eq6_averageA_w+1 definish}
\end{equation}
where the inequality is obvious by sense; also note that
\begin{align}\label{eq6_>=1}
\sum\limits_{j=1}^{T_{w}^\Sigma}\gamma _{w,j}^{2}\ge\sum\limits_{j=1}^{T_{w}^\Sigma}\gamma _{w,j}=wU_w.
\end{align}

We denote a lower estimate of $\Delta _{w}^{\text{aver}}(\K_{w})$,
see Lemma \ref{lem6_lower bound aver} below, as follows:
\begin{align}
& \Delta _{w}^{\text{rigor}}(\K_{w}):=\max \left\{1,\frac{wU_{w}}{
\theta_{N-1,q}+1-w}-w+1\right\}=  \label{eq6_estimate} \\
& =\left\{
\begin{array}{ccc}
\frac{wU_{w}}{\theta_{N-1,q}+1-w}-w+1 & \mbox{ if } & U_{w}\ge \theta_{N-1,q}+1-w,\smallskip  \\
1 & \mbox{ if } & U_{w}<\theta_{N-1,q}+1-w.
\end{array}
\right.   \notag
\end{align}

\begin{lemma}
\label{lem6_lower bound aver} For any $w$-cap $\K
_{w}\in \mathbf{S}(U_{w})$,  the following  holds:
\begin{itemize}
  \item This   inequality always fulfills
\begin{equation}
\Delta _{w}^{\text{\emph{aver}}}(\K_{w})\geq
 \Delta _{w}^{\text{\emph{rigor}}}(\K_{w}).  \label{eq6_lower bound}
\end{equation}
  \item
In \eqref{eq6_lower bound}, we have the equality
\begin{equation}\label{eq6_lower bound_equality}
\Delta _{w}^{\text{\emph{aver}}}(\K_{w})=
\Delta _{w}^{\text{\emph{rigor}}}(\K_{w})=\frac{wU_{w}}{\theta_{N-1,q}+1-w}-w+1
\end{equation}
if and only if every unisecant contains the same number $\frac{U_{w}}{\theta_{N-1,q}+1-w}$ of uncovered points where $\frac{U_{w}}{\theta_{N-1,q}+1-w}$ is integer.
  \item
In \eqref{eq6_lower bound}, the equality
\begin{equation}\label{eq6_lower bound_equality_1}
\Delta _{w}^{\text{\emph{aver}}}(\K_{w})=\Delta _{w}^{\text{\emph{rigor}}}(\K_{w})=1
\end{equation}
holds if and only if each unisecant contains at most an one uncovered
point.
\end{itemize}
\end{lemma}

\begin{proof}
By Cauchy--Schwarz--Bunyakovsky inequality, it holds that
\begin{equation}
\left( \sum\limits_{j=1}^{T_{w}^{\Sigma }}\gamma _{w,j}\right) ^{2}\leq
T_{w}^{\Sigma }\sum\limits_{j=1}^{T_{w}^{\Sigma }}\gamma _{w,j}^{2}
\label{eq6_CSBineq}
\end{equation}
where equality holds if and only if all $\gamma _{w,j}$ coincide. In this case $\gamma _{w,j}=\frac{U_{w}}{\theta_{N-1,q}+1-w}$ for all $j$ and, moreover, the ratio $\frac{U_{w}}{\theta_{N-1,q}+1-w}$ is integer. Now, by
\eqref{eq6_total unisecants}, \eqref{eq6_Gamma}, we have
\begin{equation*}
\frac{wU_{w}}{\theta_{N-1,q}+1-w}\leq \frac{\sum\limits_{j=1}^{T_{w}^{
\Sigma }}\gamma _{w,j}^{2}}{U_{w}}
\end{equation*}
that together with \eqref{eq6_Gamma}, \eqref{eq6_averageA_w+1 definish}, \eqref{eq6_>=1}, \eqref{eq6_estimate} gives
\eqref{eq6_lower bound}--\eqref{eq6_lower bound_equality_1}.
\end{proof}

\begin{remark}
\label{rem6_averunisecant} One can treat the estimate \eqref{eq6_lower bound},
\eqref{eq6_lower bound_equality} as
follows. A bundle contains $w$ unisecants having a common point, its head.
Therefore the average number of uncovered points in a bundle is
$w\gamma_{w}^{\text{aver}}-(w-1)$ where $\gamma_{w}^{\text{aver}}$ is defined in \eqref{eq6_aver unisecants} and the term $w-1$ takes into account the common point.
\end{remark}

     It is clear that for any $w$-cap $\K
_{w}\in \mathbf{S}(U_{w})$ we have
\begin{equation}\label{eq6_max>=aver}
    \max\limits_{A_{w+1}}\Delta _{w}(A_{w+1})\ge\left\lceil\Delta _{w}^{\text{aver}}(\K_{w})\right\rceil.
\end{equation}

\begin{corollary}\label{cor6_1} It hold that
    \begin{equation*}
    \max\limits_{A_{w+1}}\Delta _{w}(A_{w+1})\ge\max \left\{1,\left\lceil\frac{wU_{w}}{
\theta_{N-1,q}+1-w}-w+1\right\rceil\right\}.
\end{equation*}
\end{corollary}

\begin{remark}\label{rem6_incident}
The results and approaches, connected with estimates of line-point incidences (see e.g. \cite{MurPetr2016,MurPetrShkr2017} and the references therein) could be useful for estimates and bounds considered in this paper.
\end{remark}

Let $D\ge1$ be a constant independent of $q$. Throughout the paper we denote
\begin{align*}
&\Phi_{w,q}(D)=\frac{D(w-1)\theta_{N,q}(\theta_{N-1,q}+1-w)}{Dw\theta_{N,q}-(\theta_{N-1,q}+1-w)(w(q-1)+1)},\\
&\Upsilon_{w,q}(D)=\frac{D\theta_{N,q}}{w(q-1)+1}.
\end{align*}

\begin{lemma}
\label{lem6_phi_psi} Let $D\ge1$ be a constant independent of $q$. Let an one of the following two conditions hold:
   \begin{align*}
U _{w}\ge \Phi_{w,q}(D),\quad  \Upsilon_{w,q}(D)\ge U _{w}.
\end{align*}
Then, for any cap $\K_{w}$ of $\mathbf{S}(U_{w})$, it holds that
\begin{align*}
\Delta _{w}^{\text{\emph{aver}}}(\K_{w})\geq \frac{\mathbf{E}_{w,q}}{D}.
\end{align*}
\end{lemma}

\begin{proof}
By  \eqref{eq6_estimate}, \eqref{eq6_lower bound}, we have
\begin{equation*}
\Delta_{w}^{\text{aver}}(\K_{w})\geq \Delta_{w}^{\text{rigor}}(\K_{w})
\ge\frac{wU_{w}}{\theta_{N-1,q}+1-w}-w+1.  \label{eq6__blow}
\end{equation*}
It is easy to see that under condition $U _{w}\ge \Phi_{w,q}(D)$ it holds that
\begin{equation*}
\frac{wU_{w}}{\theta_{N-1,q}+1-w}-w+1- \frac{(w(q-1)+1)U_{w}}{
D\theta_{N,q}}\ge 0.
\end{equation*}
If  $U _{w}\le \Upsilon_{w,q}
(D)$ then $\frac{\mathbf{E}_{w,q}}{D}\le1$. On the other side, by
\eqref{eq6_estimate}, \eqref{eq6_lower bound}, we always have
$\Delta_{w}^{\text{aver}}(\K_{w})\geq \Delta_{w}^{\text{rigor}}(\K_{w})\ge1.$
\end{proof}

From Lemmas \ref{lem6_lower bound aver} and \ref{lem6_phi_psi} we obtain the corollary.

\begin{corollary}\label{cor6}
Let $D\ge1$ be a constant independent of $q$. Let an one of the following two conditions hold:
   \begin{align*}
U _{w}\ge \Phi_{w,q}(D),\quad  \Upsilon_{w,q}(D)\ge U _{w}.
\end{align*}
 Then,
 for any cap $\K_{w}$ of $\mathbf{S}(U_{w})$, there exists
an uncovered point $A_{w+1}$ providing the inequality
\begin{align*}
\Delta _{w}(A_{w+1})\geq
\frac{\mathbf{E}_{w,q}}{D}=\frac{(w(q-1)+1)U_{w}}{
D\theta_{N,q}}.
\end{align*}
\end{corollary}

\begin{proof}
By the definition of the average value \eqref{eq6_averageA_w+1 definish},
always there is an uncovered point $A_{w+1}$ providing the inequality $
\Delta _{w}(A_{w+1})\geq \Delta _{w}^{\text{aver}}(\K_{w})$, see also~\eqref{eq6_max>=aver}.
\end{proof}

\section{On reasonableness of Conjecture \ref{Conj_main}}\label{sec_reason}

In this section we  show (by reflections, calculations and figures) that in the steps of the iterative process when the rigorous estimates give not good results, actually these estimates do not reflect the real situation effectively.

$\bullet$ In the first we will illustrate the following:  when the rigorous bound \eqref{eq6_estimate}--\eqref{eq6_lower bound} is smaller than the expectation $\mathbf{E}_{w,q}$, in fact, the average value $\Delta _{w}^{\text{aver}}(\K_{w})$ of \eqref{eq6_averageA_w+1 definish} is greater (and the maximum value $\max\limits_{A_{w+1}}\Delta
_{w}(A_{w+1})$ is essentially greater) than $\mathbf{E}_{w,q}$, see Fig. \ref{fig_reasonable}.

\begin{figure}[htbp]
\includegraphics[width=\textwidth]{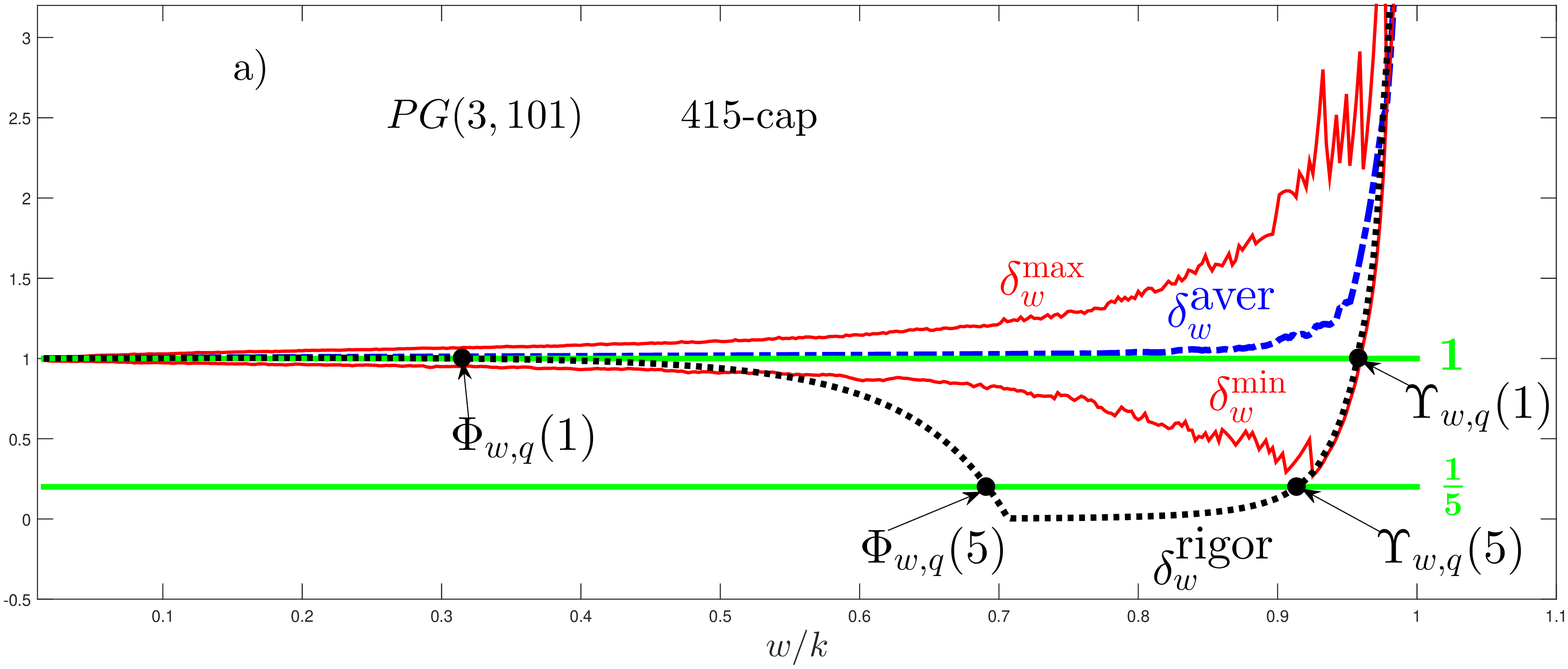}
\includegraphics[width=\textwidth]{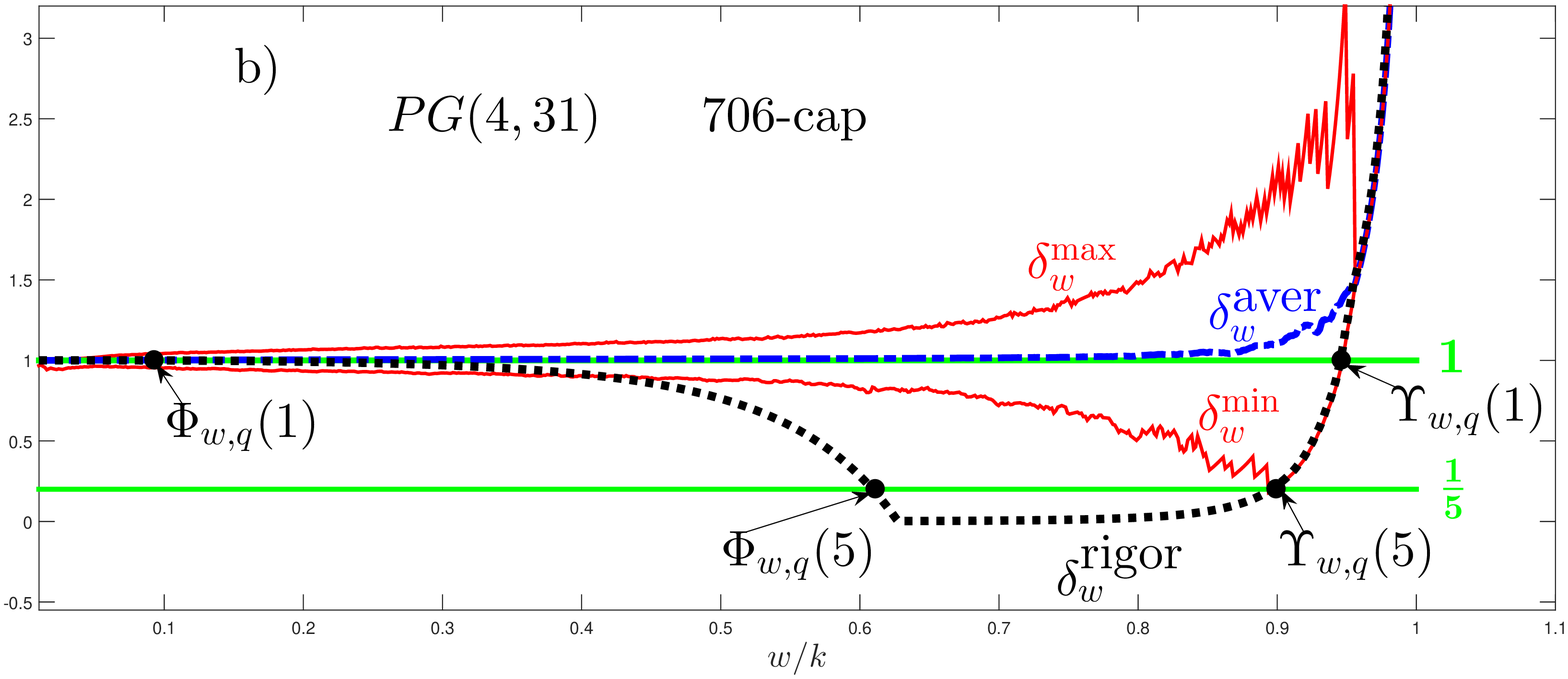}
\caption{\textbf{Illustration of reasonableness of Conjecture \ref{Conj_main}.}
Values $\delta^{\bullet}_{w}$ for a complete $k$-cap in $\PG(N,q)$. a) $N=3$, $q=101$, $k=415$;
b) $N=4$, $q=31$, $k=706$:
$\delta^{\max}_{w}$ (\emph{top solid red curve}),
$\delta^{\text{aver}}_{w}$ (\emph{the \emph{2}-nd dashed-dotted blue curve}),
$\delta^{\min}_{w}$ (\emph{the \emph{3}-rd solid red curve}),
$\delta^{\text{rigor}}_{w}$ (\emph{bottom dotted black curve}), \emph{green lines}
$y=1$ (for $D=1$) and $y=\frac{1}{5}$ (for $D=5$). The region where Conjecture \ref{Conj_main} is rigorously proved lies on the left
of $\Phi_{w,q}(D)$ and on the right of  $\Upsilon_{w,q}(D)$}
\label{fig_reasonable}
\end{figure}

We have calculated
the values $\Delta_{w}(A_{w+1})$, defined
in \eqref{eq2_Delta_w}, for
numerous concrete iterative processes  in $\PG(3,q)$ and  $\PG(4,q)$.
 It is important that \emph{for all the calculations
have been done}, it holds that
\begin{align*}
\max\limits_{A_{w+1}}\Delta
_{w}(A_{w+1})>\mathbf{E}_{w,q}.
\end{align*}
Moreover, the ratio $\max\limits_{A_{w+1}}\Delta
_{w}(A_{w+1})/\mathbf{E}_{w,q}$
has the increasing trend when $w$ grows. Thus, the variance of the
random value $\Delta_{w}$ helps to get good results.

The existence of points $A_{w+1}$ providing $\Delta _{w}(A_{w+1})>\mathbf{E}
_{w,q}$ is used by the greedy algorithms to obtain complete caps smaller
than the bounds following from Conjecture~\ref{Conj_main}.

 An
illustration of the aforesaid is shown on Fig.\,\ref{fig_reasonable}
where for complete $k$-caps in $\PG(3,101)$, $k=415$, and in $\PG(4,31)$, $k=706$,
obtained by the greedy algorithm, the values  \begin{align*}
&\delta^{\min}_{w}=\frac{\min\limits_{A_{w+1}}\Delta
_{w}(A_{w+1})}{\mathbf{E}_{w,q}},\quad \delta^{\max}_{w}=\frac{\max\limits_{A_{w+1}}\Delta
_{w}(A_{w+1})}{\mathbf{E}_{w,q}},\bigskip\bigskip\\
&\delta_{w}^{\text{aver}} =\frac{\Delta _{w}^{\text{aver}}(\K_{w})}{
\mathbf{E}_{w,q}},\quad \delta_{w}^{\text{rigor}}=\frac{\Delta _{w}^{\text{rigor}}(
\K_{w})}{\mathbf{E}_{w,q}},
\end{align*}
are presented.
The horizontal axis shows the values of $\frac{w}{k}$. The final region of the
iterative process when $U_{w}\leq \Upsilon_{w,q}(D)$ and
$\frac{\mathbf{E}_{w,q}}{D}\le1$ is shown not completely. The green lines $y=1$  and $y=\frac{1}{5}$
correspond, respectively, to Conjecture \ref{Conj_main}(ii), where $D=1$, and Conjecture~\ref{Conj_main}(i) with $D=5$. The signs~$\bullet$ correspond to the
values $\Phi_{w,q}(D)$ and $\Upsilon_{w,q}(D)$ with $D=1$ and $D=5$.
It is interesting (and expected) that, for   all the steps
of the iterative process,  we have
$\Delta_{w}^{\text{aver}}(\K_{w})>
\mathbf{E}_{w,q}$, i.e. $\delta^{\text{aver}}_{w}>1$.

 In Fig.\,\ref{fig_reasonable}, the region where we rigorously prove Conjecture \ref{Conj_main} lies on the left
of $\Phi_{w,q}(D)$ and on the right of  $\Upsilon_{w,q}(D)$. This region in $\PG(3,101)$ takes $\sim 35\%$ of the whole iterative process for
$D=1$ and $\sim 75\%$ for $D=5$.

Note that the forms of curves $\delta^{\max}_{w}$ and
$\delta^{\text{aver}}_{w}$ are similar for all $q$'s and $N$'s for which we
calculated these values.

$\bullet$ Now we consider \emph{the dispersion of the number of uncovered points on unisecants.}

 The lower estimate in \eqref{eq6_lower bound}
based on \eqref{eq6_CSBineq} is attained in two cases: either \emph{every unisecant contains the same
number of uncovered points} or \emph{each unisecant contains at most an one uncovered
point}.

The 1-st situation holds in the first steps of the
iterative process only. Then the differences $\gamma _{w,j}-\gamma
_{w,i}$ become nonzero. But, while the inequality $U_{w}(D)\geq  \Phi_{w,q}(D)$
holds, these differences are relatively small and estimate
\eqref{eq6_lower
bound} works \textquotedblleft well\textquotedblright . When $U_{w}$
decreases, the differences relatively increase, and the estimate becomes
worse in the sense that actually $\Delta _{w}^{\text{aver}}(\K_{w})$
is considerably greater than $\Delta _{w}^{\text{rigor}}(\K_{w})$.

The 2-nd situation is possible, in principle, when $U_{w}\le \theta_{N-1,q}+1-w$ and
the average number $\gamma_{w}^{\text{aver}}$  of uncovered points on an unisecant is smaller than one, see
\eqref{eq6_aver unisecants}. But on this stage of the iterative process
variations in the values $\gamma _{w,j}$ are relatively big; and
again  $\Delta _{w}^{\text{aver}}(\K_{w})$
is considerably greater than $\Delta _{w}^{\text{rigor}}(\K_{w})$.

In the final region of the iterative process, where $U _{w}\le \Upsilon_{w,q}(D)$
 and $\frac{\mathbf{E}_{w,q}}{D}\le1$, estimate
\eqref{eq6_lower bound} becomes reasonable once more. Thus, in the region
\begin{equation*}
\Phi_{w,q}(D)>U_{w}>\Upsilon_{w,q}(D)
\end{equation*}
the lower estimate \eqref{eq6_lower bound} does not reflect the real
situation effectively. This leads the necessity to formulate Conjecture \ref
{Conj_main} as a (plausible) hypothesis.

Let $\gamma_{w}^{\text{aver}}$ be defined in \eqref{eq6_aver unisecants}.
Let $\gamma_w^{\max}$ and $\gamma_w^{\min}$ be, respectively, the maximum and minimum of
the number $\gamma _{w,j}$ of uncovered points on an unisecant, i.e.
\begin{align*}
   \gamma_w^{\max}=\max_j\gamma _{w,j},\quad
   \gamma_w^{\min}=\min_j\gamma _{w,j}.
\end{align*}

An illustration of the fact that the numbers $\gamma _{w,j}$ of uncovered points on unisecants
lie in a relatively wide region is shown on Fig.\,\ref{fig_unisecant_gamma},
where for complete $k$-caps in $\PG(3,101)$, $k=415$, and in $\PG(4,31)$, $k = 706$,
obtained by the greedy algorithm, the values $\gamma_w^{\max}/\gamma_w^{aver}$ and
$\gamma_w^{\min}/\gamma_w^{aver}$ are presented.  The horizontal axis shows the values of $\frac{w}{k}$.
The such curves
were obtained for
numerous concrete iterative processes  in $\PG(3,q)$ and  $\PG(4,q)$.
 It is important that \emph{for all the calculations
have been done}, the forms of the curves are similar. Moreover, the value $\gamma_w^{\max}/\gamma_w^{aver}$
increases when the ratio $\frac{w}{k}$ grows; in the region $0.78<\frac{w}{k}<0.95$ (it is not shown in
Fig.\,\ref{fig_unisecant_gamma}); the value $\gamma_w^{\max}/\gamma_w^{aver}$ increases from $20$
to $590$ for the $415$-cap in $\PG(3,101)$ and from $36$ to $1400$ for the $706$-cap in $\PG(4,31)$.

\begin{figure}[htbp]
\includegraphics[width=\textwidth]{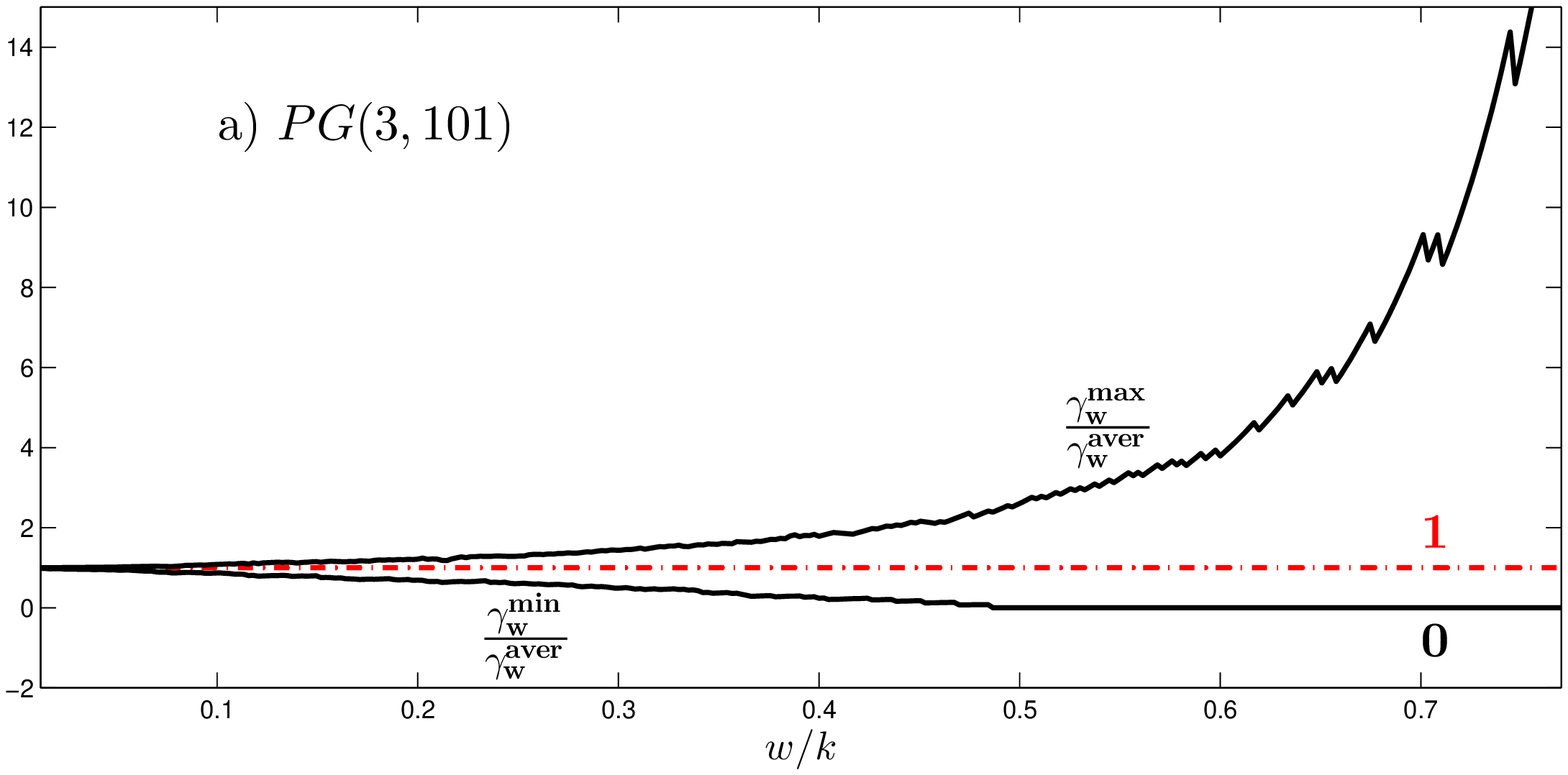}
\includegraphics[width=\textwidth]{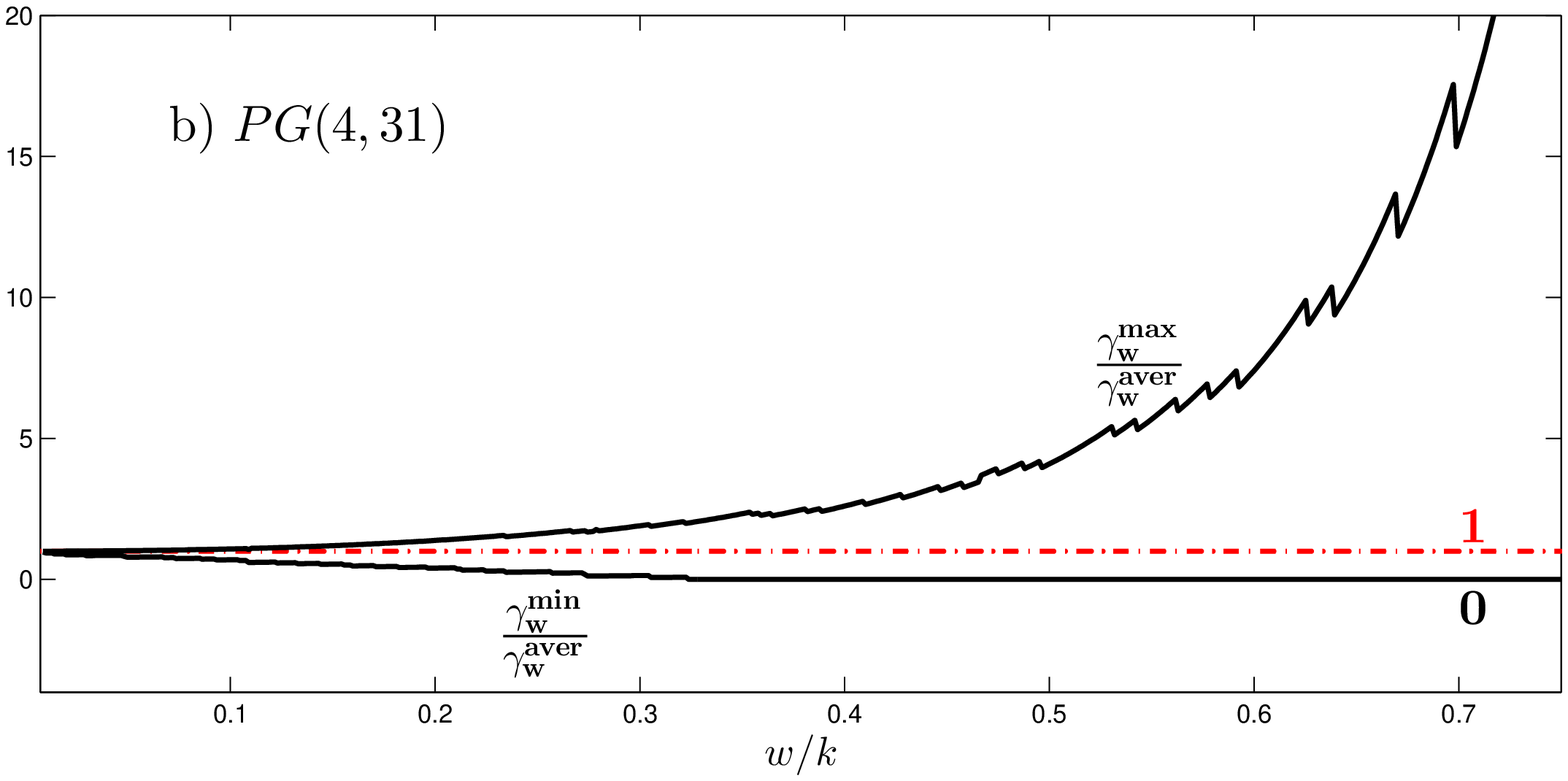}
\caption{\textbf{Dispersion of the number $\gamma _{w,j}$ of uncovered points on unisecants.}
Values $\gamma_w^{\max}/\gamma_w^{aver}$ (\emph{top solid black curve}) and
$\gamma_w^{\min}/\gamma_w^{aver}$ (\emph{bottom solid black curve}) and \emph{dashed-dotted red line} $y=1$
 for a complete $k$-cap in $\PG(N,q)$.  a) $N = 3$, $q = 101$, $k=415$; b)
$N = 4$, $q = 31$, $k = 706$}
\label{fig_unisecant_gamma}
\end{figure}

\begin{remark}\label{rem6_=->!=}
It can be proved rigorously (using Observation \ref{observ}) that if in some step of the
iterative process every unisecant contains the same
number of uncovered points then in the next step this situation does not hold.
\end{remark}

The calculations mentioned in this section and
Figs.\,\ref{fig_reasonable}, \ref{fig_unisecant_gamma}
illustrate the soundness of the key Conjecture \ref{Conj_main}.

\section{Conclusion}\label{sec_Conclus}

In the present paper, we make an attempt to obtain a theoretical upper bound on $t_2(N,q)$ with
the main term of the form $cq^\frac{N-1}{2}\sqrt{\ln q}$, where $c$ is a small constant independent of $q$. The bound is
based on explaining the mechanism of a step-by-step greedy algorithm for constructing complete
caps in $\PG(N,q)$ and on quantitative estimations of the algorithm. For a part of steps
of the iterative process, these estimations are proved rigorously. We make a natural (and wellfounded)
conjecture that they hold for other steps too. Under this conjecture we give new upper bounds on
$t_{2}(N,q)$ in the needed form, see \eqref{eq1_bnd_mainD}, \eqref{eq1_bnd_mainD=1}. We illustrate the effectiveness of the new bounds comparing them with the results of computer search from the papers \cite{BDKMP-arXivPG3q4q,BDKMP-PG3q4qENDM}, see Fig.\,\ref{fig_3_4q_1}.

We did not obtain a rigorous proof for
precisely the part of the process where the variance of the random variable $\Delta_w(A_{w+1})$ determining
the estimates implies the existence of points $A_{w+1}$ which are considerably better than what is
necessary for fulfillment of the conjecture (see the curve $\delta^{\max}_w$ in Fig.\,\ref{fig_reasonable}). In other words,
 in the steps of the iterative process when the rigorous estimates give not well results, in fact, these estimates do not reflect the real
situation effectively. The reason is that the rigorous estimates assume that the number of uncovered points on unisecants is the same for all unisecants.
However, in fact, there is a dispersion of the number of uncovered points on unisecants, see Section~\ref{sec_reason}. Moreover,  this dispersion
grows in the iterative process. So,
Conjecture \ref{Conj_main} seems to be reasonable.


\begin{thebibliography}{99}
\bibitem{ABGP_2014_JCD} N. Anbar, D. Bartoli, M. Giulietti, I. Platoni, Small complete caps from singular
cubics. \emph{J. Combin. Des.} \textbf{22}, 409--424 (2014)

\bibitem{ABGP_2015_JAC} N. Anbar, D. Bartoli, M. Giulietti, I. Platoni, Small complete caps from singular
cubics, II. \emph{J. Algebraic Combin.} \textbf{41}, 185--216 (2015)

\bibitem{AG_2013} N. Anbar, M. Giulietti, Bicovering arcs and small complete caps from elliptic curves, \emph{J.
Algebraic. Combin.} \textbf{38}, 371-392 (2013)

\bibitem{BDFKMP-PIT2014} D. Bartoli, A.A. Davydov, G. Faina,
    A.A. Kreshchuk, S. Marcugini, F.~Pambianco,
    Upper bounds on the smallest size  of a complete arc in $\PG(2,q)$
 under a certain probabilistic conjecture. \emph{Problems
 Inform. Transmission} \textbf{50}, 320--339 (2014)

\bibitem{BDFKMP-JG2015-2} D. Bartoli, A.A. Davydov, G. Faina,
    A.A. Kreshchuk, S. Marcugini, F.~Pambianco,
    Upper bounds on the smallest size of a complete arc in a
finite Desarguesian projective plane based on computer search.
\emph{J. Geom.} \textbf{107}, 89--117 (2016)

\bibitem{BDFMP-OC2013} D. Bartoli, A.A. Davydov, G. Faina,
    S. Marcugini, F.~Pambianco,  New upper bounds on the smallest size of a complete cap in the
space $\PG(3,q)$. In: \emph{Proc. VII Int. Workshop on Optimal Codes
and Related Topics, OC2013}, Albena, Bulgaria, pp. 26--32 (2013) http://www.moi.math.bas.bg/oc2013/a4.pdf

\bibitem{BDFMP-DM} D. Bartoli, A.A. Davydov, G. Faina,
    S. Marcugini, F.~Pambianco, On sizes of complete arcs in $\PG(2,q)$, \emph{Discrete Math.} \textbf{312}, 680--698  (2012)

\bibitem{BDFMP-JG2015}  D. Bartoli, A.A. Davydov, G. Faina,
    S. Marcugini, F.~Pambianco, New types of estimates for the smallest size of complete
arcs in a finite Desarguesian projective plane, \emph{J. Geom.} \textbf{106}, 1--17 (2015)

\bibitem{BDFMP-Bulg2016ENDM}  D. Bartoli, A.A. Davydov, G. Faina,
    S. Marcugini, F.~Pambianco, Conjectural upper bounds on the smallest size
of a complete cap in $\PG(N,q)$, $N\ge3$, \emph{Electron. Notes Discrete Math.} \textbf{57}, 15--20 (2017)

\bibitem{BDKMP-arXivPG3q4q}D. Bartoli, A.A. Davydov,
    A.A. Kreshchuk, S. Marcugini, F.~Pambianco, Tables, bounds and graphics of the smallest known sizes of complete caps in the spaces PG(3,q) and PG(4,q),
    arXiv:1610.09656[math.CO] (2016)
     http://arxiv.org/abs/1610.09656

\bibitem{BDKMP-PG3q4qENDM}D. Bartoli, A.A. Davydov,
    A.A. Kreshchuk, S. Marcugini, F.~Pambianco, Upper bounds on the smallest size of a
complete cap in $\PG(3,q)$ and $\PG(4,q)$,  \emph{Electron. Notes Discrete Math.} \textbf{57}, 21--26 (2017)

\bibitem{BFG2013} D. Bartoli, G. Faina, M. Giulietti, Small complete caps in three-dimensional Galois spaces,
 \emph{Finite Fields Appl.} \textbf{24}, 184--191 (2013)

\bibitem{BMP2015} D. Bartoli, G. Faina, S. Marcugini, F.~Pambianco, A
    construction of small complete caps in projective spaces, \emph{J. Geom.},  \textbf{108},   215-246 (2017)

\bibitem{BGMP-ScatCap}  D. Bartoli, M. Giulietti, G. Marino, O. Polverino,
Maximum scattered linear sets and complete caps in Galois spaces,
 \emph{Combinatorica}, to appear.

\bibitem{BMP-quantum} D. Bartoli,  S. Marcugini, F.~Pambianco,  New quantum caps in $\mathrm{PG}(4,4)$, \emph{J.
    Combin. Des.} \textbf{20}, 448--466 (2012)

\bibitem{BrinkBirthday} D. Brink, A (probably) exact solution to the
birthday problem, \emph{The Ramanujan J.} \textbf{28}, 223--238 (2012)

\bibitem{BLP1998}
R.A. Brualdi, S. Litsyn, V.S. Pless, Covering radius.
In: Pless, V.S., Huffman, W.C., Brualdi, R.A. (eds) \emph{Handbook of Coding Theory,}
 Vol.~1, pp. 755--826. Elsevier, Amsterdam, The Netherlands (1998)

\bibitem{BirthMajor} M.L. Clevenson, W. Watkins, Majorization and
the Birthday inequality, \emph{Math. Magazine} \textbf{64}, 183--188 (1991)

\bibitem{CHLL1997}
G.D. Cohen, I.S. Honkala, S. Litsyn, A.C. Lobstein,
\emph{Covering Codes},  Elsevier, Amsterdam, The Netherlands (1997))

\bibitem{DFMP-JG2009} A.A. Davydov, G. Faina,
    S. Marcugini, F.~Pambianco,  On sizes of complete caps in projective spaces $\PG(n,q)$ and arcs in
planes $\PG(2,q)$, \emph{J. Geom.} \textbf{94}, 31--58 (2009)

\bibitem{DGMP2010}
A.A. Davydov, M. Giulietti, S. Marcugini, F.~Pambianco, New inductive
  constructions of complete caps in $\PG(n,q)$, $q$ even. \emph{J. Combin.
  Des.} {\bf 18}, 177--201 (2010)

\bibitem{DMP-JG2004} A.A. Davydov, S. Marcugini, F.~Pambianco,
Complete caps in projective spaces $\PG(n,q)$, \emph{J. Geom.} \textbf{80}, 23--30 (2004)

\bibitem{DavOst} A.A. Davydov, P.R.J. \"{O}sterg\aa rd, Recursive constructions of complete caps, \emph{J. Statist.
Planning. Infer.} \textbf{95}, 167--173 (2001)

\bibitem{FaPasSch_2012}   G. Faina, F. Pasticci, L. Schmidt,
Small complete caps in Galois spaces.
 \emph{Ars Combin.} \textbf{105}, 299--303 (2012)

  \bibitem{GDT1991}
E.M. Gabidulin, A.A. Davydov, L.M. Tombak, Linear codes with covering radius
  $2$ and other new covering codes, \emph{IEEE Trans. Inform. Theory} {\bf 37},
  219--224  (1991)

\bibitem{GiuliettiAffin}  M. Giulietti, Small complete caps in Galois affine spaces, \emph{J. Algebraic Combin.} \textbf{25}(2),
149--168 (2007)

  \bibitem{Giulietti2007}
M. Giulietti, Small complete caps in $\PG(n,q)$, $q$ even,
\emph{J. Combin. Des.} {\bf 15}, 420--436  (2007)

\bibitem{GiuliettiSurvey} M. Giulietti, The geometry of covering codes: small complete caps and saturating sets in
Galois spaces, \emph{Surveys in Combinatorics} 2013 - London Mathematical Society Lecture Note
Series 409, Cambridge University Press, 2013, pp. 51--90.

\bibitem{GiulPast} M. Giulietti, F. Pasticci, Quasi-Perfect Linear Codes with Minimum Distance 4,
\emph{IEEE Trans.  Inform. Theory} \textbf{53}(5), 1928--1935 (2007)

\bibitem{HirsSt-old} J.W.P. Hirschfeld, L. Storme, The
    packing problem in statistics, coding theory and finite projective spaces,
     \emph{J. Statist. Planning Infer.} \textbf{72}, 355--380 (1998)

\bibitem{HirsStor-2001} J.W.P. Hirschfeld, L. Storme,
    The
    packing problem in statistics, coding theory and finite geometry: update 2001.
     In:  A. Blokhuis, J.W.P. Hirschfeld, et al. (eds.) \emph{Finite
    Geometries, Developments of Mathematics}, vol. 3, Proc. of the
    Fourth Isle of Thorns Conf., Chelwood Gate, 2000, pp. 201--246.
    Kluwer Academic Publisher, Boston (2001)

\bibitem{HirsThas-2015} J.W.P. Hirschfeld, J.A. Thas,
 Open    problems in    finite projective    spaces,   \emph{Finite Fields
    Their Appl.} \textbf{32},  44--81 (2015)

\bibitem{KV} J.H. Kim, V. Vu, Small complete arcs in
    projective planes. \emph{Combinatorica} \textbf{23}, 311--363 (2003)

\bibitem{LandSt}  I. Landjev, L. Storme,  Galois geometry and coding theory. In: \emph{Current Research Topics in Galois geometry,} J. De Beule, L. Storme, Eds.,
Chapter 8,  Nova Science Publisher,  (2011) pp. 185--212.

\bibitem{MurPetr2016} B. Murphy, G. Petridis, A point-line incidence identity in finite fields, and applications,
\emph{Moscow J. Combinatorics Number Theory} \textbf{6}, 64--95 (2016)

\bibitem{MurPetrShkr2017} B. Murphy, G. Petridis, O. Roche-Newton, M. Rudnev,  I.D. Shkredov,
New results on sum-product type growth over fields,
arXiv:1702.01003v2[math.CO] (2017)

\bibitem{PS1996}
F. Pambianco, L. Storme, Small complete caps in spaces of even characteristic,
\emph{J. Combin. Theory Ser. A} {\bf 75}, 70--84 (1996)

\bibitem{Platoni} I. Platoni, Complete caps in $\AG(3,q)$ from
    elliptic curves,\emph{ J.  Alg. Appl.} {\bf 13}, 1450050 (8 pages) (2014)

\bibitem{BirthRevis} M. Sayrafiezadeh, The Birthday problem revisited,
\emph{Math. Magazine} \textbf{67}, 220--223  (1994)

\bibitem{Segre} B. Segre, On complete caps and ovaloids in three-dimensional Galois spaces of characteristic
two, \emph{Acta Arith.} \textbf{5}, 315--332 (1959)

\bibitem{szoT93} T. Sz\H{o}nyi, Arcs, caps, codes and
    3-independent subsets. In: G. Faina et al.  (eds.)
    \emph{Giornate di Geometrie Combinatorie}, Universit\`{a} degli
    studi di Perugia, pp. 57--80. Perugia (1993)

\bibitem{Tonchev}  V.D. Tonchev, Quantum codes from caps,
    \emph{Discrete Math.} {\bf 308}, 6368--6372  (2008)
\end{thebibliography}
\end{document}